\documentclass[12pt]{article}
\usepackage{amsmath,amssymb,amsthm}

\newtheorem{theorem}{Theorem}[section]
\newtheorem{lemma}[theorem]{Lemma}
\theoremstyle{definition}
\newtheorem{definition}[theorem]{Definition}
\newtheorem*{question}{Question}
\numberwithin{equation}{section}

\begin{document}

\title{On the distortion of knots on embedded surfaces}

\author{John Pardon}

\date{18 December 2010; Revised 28 April 2011}

\maketitle

\begin{abstract}
Our main result is a nontrivial lower bound for the distortion of some specific knots.  In particular, we show that the distortion 
of the torus knot $T_{p,q}$ satisfies $\delta(T_{p,q})\geq\frac 1{160}\min(p,q)$.  This answers a 1983 question of Gromov.
\end{abstract}

\section{Introduction}

If $\gamma$ is a rectifiable curve in $\mathbb R^3$, then its \emph{distortion} is defined to be the quantity:
\begin{equation}\label{distdef}
\delta(\gamma)=\sup_{p,q\in\gamma}\frac{d_{\gamma}(p,q)}{d_{\mathbb R^3}(p,q)}\geq 1
\end{equation}
where $d_\gamma$ denotes the arclength along $\gamma$ and $d_{\mathbb R^3}$ denotes the Euclidean distance in 
$\mathbb R^3$.  In 1983, Gromov asked the following question:

\begin{question}[{quoted directly from \cite[p114]{gromov2}}]
Does \emph{every} isotopy class of knots in $\mathbb R^3$ have a representative in $\mathbb R^3$ with distortion $<100$?  
Is it so for all torus knots $T_{p,q}$ for $p,q\to\infty$?
\end{question}

In this paper we show that this is not the case.  For a knot $K$, let $\delta(K)$ denote the infimum of $\delta(\gamma)$ over all 
rectifiable curves $\gamma$ in the isotopy class $K$.

\begin{theorem}\label{torusthm}
Let $T_{p,q}$ denote the $(p,q)$-torus knot.  Then $\delta(T_{p,q})\geq\frac 1{160}\min(p,q)$.
\end{theorem}

\begin{theorem}\label{cablingthm}
Let $K$ be a nontrivial tame knot, and let $K_{p,q}$ denote the $(p,q)$-cabling of $K$, where $p$ is 
the longitudinal coefficient.  Then $\delta(K_{p,q})\geq\frac 1{160}p$.
\end{theorem}

These are both consequences of the following more general theorem which deals with knots lying on any embedded surface.

\begin{theorem}\label{surfacethm}
Let $F\subseteq\mathbb R^3$ be a PL embedded closed surface of genus $g\geq 1$.  Let $\mathcal S$ denote the set of nontrivial 
isotopy classes of simple loops on $F$, and let $i:\mathcal S\times\mathcal S\to\mathbb Z_{\geq 0}$ denote the 
minimum geometric intersection number.  Let $\beta\in\mathcal S$, and let $K_\beta$ denote the corresponding 
knot in $\mathbb R^3$.  Then we have:
\begin{equation}
\delta(K_\beta)\geq\frac 1{160}I(F,\beta)
\end{equation}
where we define $I(F,\beta):=\min_{\alpha\in U}i(\alpha,\beta)$, where $U$ is the set of 
all $\alpha\in\mathcal S$ which bound a PL embedded disk whose interior is disjoint from $F$.
\end{theorem}

Some calculation shows that $\delta(T_{p,q})$ is bounded above by a constant times:
\begin{equation}\label{distexplicit}
\min\left\{\left|xp\right|+\left|yq\right|\Bigm|xp+yq=1,(x,y)\in\mathbb Z^2\right\}
\end{equation}
(indeed, the standard embedding on a torus of revolution of the correct dimensions achieves this).  Thus for some 
families of torus knots (e.g.\ the $(p,kp+1)$-torus knots for fixed $k$), 
the bound in Theorem \ref{torusthm} is sharp up to a constant factor.  In particular, Theorem \ref{surfacethm} gives the 
best possible asymptotics in terms of $I(F,\beta)$, up to a constant factor.  
However, to the author's knowledge, there are no known embeddings of the torus knots which 
achieve distortion smaller than a constant times (\ref{distexplicit}), which is clearly at least $\max(p,q)$.  
Thus one would like to improve Theorem \ref{torusthm} substantially.  
For instance, surely $\lim_{p\to\infty}\delta(T_{2,p})=\infty$, though at present it seems that a significant new 
idea would be needed before our methods would yield such a result.

Despite the simplicity of (\ref{distdef}), very little is known about the distortion of knots, especially if one is interested 
in lower bounds.  Gromov showed that for any simple closed curve $\gamma$, we have $\delta(\gamma)\geq\frac 12\pi$, with equality 
if and only if $\gamma$ is a circle, thus determining $\delta(\operatorname{unknot})$.  Denne and Sullivan \cite{dennesullivan} 
have shown that for any nontrivial tame knot $K$, we have $\delta(K)\geq\frac 53\pi$.

The principal difficulty in dealing with the distortion is that it is not \emph{coercive}, that is, $\delta(\gamma)<\infty$ does 
not imply nice regularity properties of $\gamma$.  As a result, even if $\gamma$ has finite distortion, 
it may still be knotted at arbitrarily small length scales.  The ``remarkable simple closed curve'' of Fox \cite{foxwild} 
is an illuminating example of a wild knot with finite distortion, and it is easy to observe (see, e.g., Gromov \cite[p308]{gromov1}) 
that there is some finite threshold of the distortion under which there are infinitely many tame knots (even prime ones).  It is also worth remarking that if we let 
$\delta_{\operatorname{PL}}(K)$ equal the infimal distortion over all polygonal representatives of $K$, 
then it is an open question whether or not $\delta_{\operatorname{PL}}(K)=\delta(K)$ for all tame knots $K$.  Because 
of the possibility of pathological embeddings with small distortion, it is also an open problem to establish the existence 
of minimizers of $\delta$ in any nontrivial knot class.  It is not clear how much information can be obtained from 
calculus of variations (e.g.\ as applied by Mullikin \cite{mullikin}), the main 
difficulty being that (\ref{distdef}) is a supremum, not an average.

It is perhaps relevant to contrast the situation when an ``energy functional'' is coercive.  For example, suppose one can derive an 
\emph{a priori} $C^{1,\alpha}$ estimate in terms of the value of an energy functional on a curve.  Since $C^{1,\alpha}$ curves are tame, 
we conclude that finite energy implies tameness.  If in addition the energy functional blows up for non-embedded curves (and depends continuously on the input curve), a straightforward 
compactness argument (Arzel\`a-Ascoli) shows the existence of energy minimizers in any tame knot class, as well as the finiteness of the number of knot classes 
with energy less than any finite threshold.  As a typical example of a functional to which this applies, we cite the knot energies $e_j^p$ of 
O'Hara \cite{ohara1,ohara2} in the range $jp>2$.  Freedman, He, and Wang \cite{freedman} have studied $e_2^1$ (the ``M\"obius energy'') in detail, where the analysis 
is much more delicate since this is the critical case $jp=2$, and only a weaker form of coercivity holds.  It is natural to interpret the limit $e_0^\infty$ 
as the logarithm of the distortion; however, as noted above, the distortion is not coercive, and thus the basic 
methods used to deal with $e_j^p$ apparently do not apply.

\subsection{Acknowledgements}

Most thanks is owed to David Gabai, whose suggestion has led to a significantly improved estimate over the author's original $\Omega([I(F,\beta)/\log I(F,\beta)]^{1/3})$, as 
well as to a much simplified proof.  The author is grateful to David Gabai and Conan Wu for their interest in the result and their discussions 
about it with the author.  The author also thanks Greg Kuperberg and John Sullivan for many useful comments on an earlier version of this paper.  
Detailed and insightful comments from the referee have also been very helpful in improving the clarity of this paper.  
The author first found out about this problem while reading a problem list compiled by Mohammad Ghomi.

\section{Proofs}

We shall prove Theorem \ref{surfacethm} by contradiction.  Specifically, we shall show that if the desired inequality is violated, 
then there exist arbitrarily small regions of $\mathbb R^3$ whose intersection with $F$ has a connected component of genus $g$ (recall $g$ is the genus of $F$).  
An outline of the proof is as follows.

We rely on two purely topological lemmas.  Perhaps the key to the proof is Lemma \ref{spherefamily}, which states, roughly 
speaking, that if one has a suitably generic family of embedded spheres $\{S_t\}_{t\in[0,1]}$, and each of these spheres has \emph{inessential} intersection 
with $F$, then the component of $F\setminus S_t$ with genus $g$ must stay in a single connected component of $\mathbb R^3\setminus S_t$ as 
$t$ varies.  A second essential fact is encapsulated in Lemma \ref{mustbeinessential}, which says that if a sphere intersects $\beta$ fewer 
than $I(F,\beta)$ times, then its intersection with $F$ must be inessential.

These two purely topological facts are combined with geometric information (i.e.\ the distortion) as follows.  Suppose that the portion 
of $F$ contained in some region of $\mathbb R^3$, say the ball $B(\mathbf 0,1)$, has a connected component of genus $g$.  
Counting the number of intersections between $\beta$ and the sphere $S(\mathbf 0,r)$ gives a function of $r$, and the integral of this function 
over $r\in[1,1+\epsilon]$ is less than or equal to the length of $\beta$ contained in $B(\mathbf 0,1+\epsilon)$.  
This length is bounded in terms of the distortion $\delta(\beta)$, so we find that there exists some $r_1\in(1,1+\epsilon)$ such 
that $\#(\beta\cap S(\mathbf 0,r_1))\leq\epsilon^{-1}\delta(\beta)$ (we'll ignore 
constants in this paragraph).  By a similar argument, we find $s_1\in(-\epsilon,\epsilon)$ such that $\#(\beta\cap\{z=s_1\}\cap B(\mathbf 0,r_1))\leq\epsilon^{-1}\delta(\beta)$.  
The crucial step is as follows: we consider a family of spheres $\{S_t\}_{t\in[0,1]}$, defined by starting with $S(\mathbf 0,r_1)$ and performing a $2$-surgery 
along the disk $\{z=s_1\}\cap B(\mathbf 0,r_1)$.  By construction, $\#(S_t\cap\beta)\leq\epsilon^{-1}\delta(\beta)$ for all $t\in[0,1]$.  
If $\epsilon^{-1}\delta(\beta)<I(F,\beta)$, then by Lemma \ref{mustbeinessential}, $S_t$ has only inessential intersections 
with $F$ for all $t\in[0,1]$.  Thus by Lemma \ref{spherefamily}, the portion of $F$ contained in one of the two half-spheres which comprise $S_1$ must have a 
connected component of genus $g$.  In summary, we have started with a region in $\mathbb R^3$ whose intersection with $F$ has a connected 
component of genus $g$, and under the assumption $\delta(\beta)<\epsilon\cdot I(F,\beta)$, we have produced a smaller region with the same property.  
Iterating this construction yields arbitrarily small such regions, a contradiction.  Thus $\delta(\beta)\geq\epsilon\cdot I(F,\beta)$ and Theorem \ref{surfacethm} is proved.

The topological portions of the proof appear first, in Sections \ref{spheresurface} and \ref{essobstr}, which contain Lemmas \ref{spherefamily} and \ref{mustbeinessential} 
respectively.  In Section \ref{dbsec}, we formalize the sphere cutting 
procedure from the previous paragraph in terms of ``double bubbles'' (this is still purely topological).  Then in 
Section \ref{geompart}, we use integral geometry as above to relate the distortion with intersection counts, and thus finish the proof of Theorem \ref{surfacethm}.

We shall do all the necessary topology in the piecewise-linear (PL) category.  To handle arbitrary topological embeddings 
of the knot, we appeal to the Moise--Bing triangulation theorem to convert the relevant topology to the PL category.  The referee 
has pointed out that by working in the 
PL category instead of the smooth category, we could avoid the use of Munkres' theorem \cite{munkres} on approximating PL 
homeomorphisms with diffeomorphisms.

\subsection{Embedded surfaces and families of spheres}\label{spheresurface}

\begin{definition}
A simple closed curve on a closed surface $\Sigma$ is called \emph{inessential} if and only if it bounds a disk in $\Sigma$; otherwise 
it is called \emph{essential}.
\end{definition}

\begin{definition}
If $S\subseteq\mathbb R^3$ is a disjoint union of embedded spheres, then we denote by $\operatorname{int}(S)$ (the ``interior'' of $S$) the collection of 
components of $\mathbb R^3\setminus S$ which are separated from infinity by an odd number of components of $S$.
\end{definition}

\begin{lemma}\label{spherefamily}
Suppose that $F\subseteq\mathbb R^3$ is a PL embedded closed surface of genus $g\geq 1$.  Let $\{S_t\}_{t\in[0,1]}$ be a PL one-parameter family 
of embedded surfaces, with the property that $S_t$ is a disjoint union of embedded spheres transverse to $F$, except for finitely many values of $t\in(0,1)$, 
when either $S_t$ or $S_t\cap F$ undergoes a single surgery.  Suppose that $S_t\cap F$ consists solely of curves inessential in $F$ whenever $S_t$ is transverse to $F$.  
If $F\cap\operatorname{int}(S_0)$ has a connected component of genus $g$, then so does $F\cap\operatorname{int}(S_1)$.
\end{lemma}

\begin{proof}
At any time $t$ when $S_t$ is a disjoint union of spheres transverse to $F$, we know that $F\cap S_t$ consists solely of inessential curves, 
and thus $F\setminus S_t$ has a connected component of genus $g$.  Furthermore, for such values of $t$, we have:
\begin{equation}
\dim_{\mathbb Q}\operatorname{im}(H_1(F\cap\operatorname{int}(S_t),\mathbb Q)\to H_1(F,\mathbb Q))=\begin{cases}2g&g(F\cap\operatorname{int}(S_t))=g\cr 0&\text{otherwise}\end{cases}
\end{equation}
Doing any of the allowed surgeries can change this value by at most $1$, so since it always equals either $2g$ or $0$, we conclude that 
it must be constant.  Thus the lemma follows.
\end{proof}

\subsection{An obstruction to essential intersections}\label{essobstr}

\begin{definition}
For a PL embedded closed surface $F\subseteq\mathbb R^3$, let $\mathcal S$ denote the set of nontrivial isotopy 
classes of simple loops on $F$, and let $i:\mathcal S\times\mathcal S\to\mathbb Z_{\geq 0}$ denote the minimum 
geometric intersection number.  For $\beta\in\mathcal S$, we define:
\begin{equation}
I(F,\beta):=\min_{\alpha\in U}i(\alpha,\beta)
\end{equation}
where $U$ is the set of all $\alpha\in\mathcal S$ which bound a PL embedded disk whose interior is disjoint from $F$.
\end{definition}

\begin{lemma}\label{mustbeinessential}
Suppose that $F\subseteq\mathbb R^3$ is a PL embedded closed surface of genus $g\geq 1$ which contains a PL simple closed curve $\beta$.  
Let $S$ be a PL embedded sphere which is transverse to $F$ and $\beta$, and which satisfies $\left|S\cap\beta\right|<I(F,\beta)$.  Then 
every curve in $S\cap F$ is inessential in $F$.
\end{lemma}

\begin{proof}
Suppose for sake of contradiction that $S\cap F$ contains an essential curve.  Consider the set of all such essential curves as a 
collection of curves on of $S$, 
and let $\alpha$ be any such curve that is innermost (i.e.\ bounds a disk in $S$ which contains no other essential curves).  
Then by definition $\alpha$ bounds a disk $D$ whose intersection with $F$ consists solely of inessential curves.  Considering such 
an inessential curve which is innermost in $F$, it is easy to see that we can modify $D$ so as to eliminate this intersection, while keeping $D$ embedded.  
Proceeding in this way, we may eliminate all the intersections of $D$ with $F$, and thus conclude that $\alpha$ bounds a 
PL embedded disk whose interior is disjoint from $F$.  Then by definition of $I(F,\beta)$, we have $\left|\alpha\cap\beta\right|\geq I(F,\beta)$, which 
contradicts the assumption $\left|S\cap\beta\right|<I(F,\beta)$.  Thus we are done.
\end{proof}

\subsection{Double bubbles with few intersections with $\beta$}\label{dbsec}

\begin{definition}
A \emph{double bubble} is a $2$-complex $(S,D)$ consisting of a sphere $S$ and a disk $D$ glued along its boundary to a simple loop in $S$.  
The curve $\partial D\subseteq S$ thus divides $S$ into two disks, and gluing either one of these disks to $D$ along their common boundary, 
we get two spheres, which we refer to as the two halves of the double bubble.

When we say that a double bubble $(S,D)$ is embedded in $\mathbb R^3$, it is required that $D$ is contained in $\operatorname{int}(S)$.
\end{definition}

\begin{lemma}\label{doublebubble}
Suppose that $F\subseteq\mathbb R^3$ is a PL embedded closed surface of genus $g\geq 1$ which contains a PL simple closed curve $\beta$.  
Let $(S,D)$ be a double bubble PL embedded in $\mathbb R^3$, and let $H_1$ and $H_2$ denote the two halves of the double bubble.  Suppose 
that $(S,D)$ is transverse to $F$ and $\beta$, and that:
\begin{equation}
\left|S\cap\beta\right|+2\left|D\cap\beta\right|<I(F,\beta)
\end{equation}
If $F\cap\operatorname{int}(S)$ has a connected component of genus $g$, then so does $F\cap\operatorname{int}(H_i)$ for some $i\in\{1,2\}$.
\end{lemma}

\begin{proof}
Let $\{S_t\}_{t\in[0,1]}$ be a PL one-parameter family of embedded spheres, defined by starting with $S_0=S$ and performing a 
$2$-surgery along $D$ as $t$ varies from $0$ to $1$.  Putting this family in general position, we see that $S_t$ 
is transverse to $F$ except for a finite number of values of $t\in(0,1)$ at which $S_t\cap F$ undergoes a 
single surgery.  By construction, we have:
\begin{align}
\left|S_t\cap\beta\right|&\leq\left|S\cap\beta\right|+2\left|D\cap\beta\right|<I(F,\beta)
\end{align}
Thus by Lemma \ref{mustbeinessential}, every curve in $S_t\cap F$ is inessential in $F$ (for 
those times $t$ when $S_t$ is transverse to $F$).  Now the hypotheses of Lemma \ref{spherefamily} are satisfied, and by assumption, $F\cap\operatorname{int}(S_0)$ has 
a connected component of genus $g$.  Thus we conclude that $F\cap\operatorname{int}(S_1)$ 
has a connected component of genus $g$.  This component is contained in the interior of one of the two spheres that 
comprise $S_1$, and these two spheres are in turn contained in $H_1$ and $H_2$ respectively.  Thus $F\cap\operatorname{int}(H_i)$ 
has a connected component of genus $g$ for some $i\in\{1,2\}$.
\end{proof}

For technical reasons in the proof of Theorem \ref{surfacethm}, we shall not be able to ensure that the relevant double bubble is PL at 
the points where it intersects $\beta$.  We thus need to prove a version of Lemma \ref{doublebubble} which allows such 
double bubbles.  We shall need the following lemma, which shows how to straighten the double bubble without increasing 
its intersections with $\beta$.

\begin{lemma}\label{straighteninglemma}
Fix $p\in\mathbb R^3$.  Suppose that $F\subseteq\mathbb R^3$ is a PL embedded surface containing $p$, and suppose that 
$\beta$ is a simple PL curve on $F$ passing through $p$.  Let $U$ be an open embedded surface such that $U\cap\beta=\{p\}$ and
$U\setminus\{p\}$ is PL.  Then there exists a perturbation $U'$ of $U$, supported in an arbitrarily small 
neighborhood of $p$, such that $U'$ is PL embedded, intersects $\beta$ at most once, and is transverse to $F$ and $\beta$ in a neighborhood of its intersection with $\beta$.
\end{lemma}

\begin{proof}
Let $T$ be a small tubular neighborhood of $\beta$ in a neighborhood of $p$, chosen so that $\partial T$ is transverse to $U$.  Any 
curve in $U\cap\partial T$ that is inessential and innermost in $\partial T$ may be eliminated by locally modifying $U$.  After all such modifications, $U\cap\partial T$ 
consists only of meridians of $\partial T$.  If $U\cap\partial T=\varnothing$, then we are done.  Otherwise, observe that any such circle which is innermost in $U$ 
must correspond to an intersection $U\cap\beta$.  Since there is only one such intersection, there is only one innermost circle.  Thus the circles $U\cap\beta$ are totally ordered by containment (they are isotopic to concentric circles).  Pick the outermost circle $\alpha$ on $U$.  This is a meridian of $\partial T$ and as such clearly bounds 
a PL disk which intersects $\beta$ once and is transverse to $F$ and $\beta$.  Gluing this disk into $U$ in place of whatever is inside $\alpha$, we are done.
\end{proof}

\begin{lemma}\label{doublebubble2}
Suppose that $F\subseteq\mathbb R^3$ is a PL embedded closed surface of genus $g\geq 1$ which contains a PL simple closed curve $\beta$.  
Let $(S,D)$ be a double bubble embedded in $\mathbb R^3$, and let $H_1$ and $H_2$ denote its two halves.  Assume that $(S,D)$ is PL away 
from its intersections with $\beta$ and that $\partial D\cap\beta=\varnothing$.  Suppose that:
\begin{equation}\label{fewint}
\left|S\cap\beta\right|+2\left|D\cap\beta\right|<I(F,\beta)
\end{equation}
If $F\cap\operatorname{int}(S)$ has a connected component of genus $g$, then so does $F\cap(B(\epsilon)+\operatorname{int}(H_i))$ for some $i\in\{1,2\}$ 
and every $\epsilon>0$ (where $B(\epsilon)+\operatorname{int}(H_i)$ denotes the $\epsilon$-neighborhood of $\operatorname{int}(H_i)$).
\end{lemma}

\begin{proof}
First, apply Lemma \ref{straighteninglemma} to each intersection of $(S,D)$ with $\beta$ to construct an arbitrarily small perturbation 
$(S',D')$ which is PL, still satisfies (\ref{fewint}), and is transverse to $F$ and $\beta$ in a neighborhood of its intersections with $\beta$.  Second, 
put $(S',D')$ in general position to achieve transversality with $F$ everywhere, and call the resulting arbitrarily small perturbation $(S'',D'')$.  
At this point, $(S'',D'')$ satisfies the hypotheses of Lemma \ref{doublebubble}.

Since $F\cap\operatorname{int}(S)$ is an open surface, the property of it having a connected component of genus $g$ is preserved 
under sufficiently small perturbations of $S$.  Thus we may assume that $F\cap\operatorname{int}(S'')$ has a connected 
component of genus $g$.  Then applying Lemma \ref{doublebubble}, we conclude that $F\cap\operatorname{int}(H_i'')$ 
has a connected component of genus $g$ for some $i\in\{1,2\}$.  Our perturbations can be arbitrarily small, so the lemma follows.
\end{proof}

\subsection{Integral geometry and the proof of Theorem \ref{surfacethm}}\label{geompart}

\begin{proof}[Proof of Theorem \ref{surfacethm}]
By hypothesis, we are given a PL curve $\beta$ on a surface $F$ PL embedded in $\mathbb R^3$.  We will prove that 
if $\beta^\ast$ is isotopic to $\beta$, then $\delta(\beta^\ast)\geq\frac 1{160}I(F,\beta)$.

By assumption, there is a homeomorphism $\psi:\mathbb R^3\to\mathbb R^3$ which sends $\beta$ to $\beta^\ast$.  It is a fundamental 
fact of three-dimensional topology (a result of Moise \cite{moise} and later Bing \cite{bing}) that homeomorphisms 
of three-manifolds can be approximated by PL homeomorphisms (see also Hamilton \cite{hamilton} for a 
modern proof based on the methods of Kirby and Siebenmann \cite{kirsie}).  Specifically, we apply \cite[p62 Theorem 8]{bing} to modify 
$\psi:\mathbb R^3\setminus\beta\to\mathbb R^3\setminus\beta^\ast$ so that it is a PL homeomorphism.  We let $F^\ast$ denote the image of $F$ under 
$\psi$.  In symbols, we have a homeomorphism:
\begin{align}
\psi:\mathbb R^3&\to\mathbb R^3\cr
(F,\beta)&\mapsto(F^\ast,\beta^\ast)
\end{align}
and the restriction $\psi:\mathbb R^3\setminus\beta\to\mathbb R^3\setminus\beta^\ast$ is PL.  This modification 
of $\psi$ is necessary in order to satisfy the hypotheses of Lemma \ref{doublebubble2}, which we will eventually apply.

Let $\operatorname{Box}(r)$ denote the set $\{|x|<r,|y|<2^{1/3}r,|z|<2^{2/3}r\}$ in $\mathbb R^3$.  We will make use of the convenient 
fact that the plane $\{z=0\}$ divides $\operatorname{Box}(r)$ into two copies of $\operatorname{Box}(2^{-1/3}r)$ (this is, however, 
not crucial for our method of proof; scaled copies of any convex set $K$ could be used in place of $\operatorname{Box}$; we just might 
have to cut $K$ into more pieces to make every piece fit in a smaller copy of $K$).  Now we define 
the (clearly nonempty) set:
\begin{align}
\mathcal R:=\left\{r>0\;\middle|\;\parbox{3.5in}{there exists an open subset of $\mathbb R^3$ congruent to $\operatorname{Box}(r)$ 
whose intersection with $F^\ast$ has a connected component of genus $g$}\right\}
\end{align}
We will show that if $\delta(\beta^\ast)<\frac 1{160}I(F,\beta)$, then $(1-\delta)\mathcal R\subseteq\mathcal R$ for some $\delta>0$.  This contradicts 
the obvious fact that $\inf\mathcal R>0$, and thus finishes the proof.

Suppose $r_0\in\mathcal R$; our aim will be to show that $(1-\delta)r_0\in\mathcal R$ for some $\delta>0$ (independent of $r_0$).  We may assume without loss of generality 
that $\operatorname{Box}(r_0)\cap F^\ast$ has a connected component of genus $g$, and that $r_0=1$.  Fix some $\epsilon>0$.  We begin with the following integral geometric estimate:
\begin{equation}
\int_1^{1+\epsilon}\#(\beta^\ast\cap\partial\operatorname{Box}(r))\,dr\leq\operatorname{Length}(\beta^\ast\cap\operatorname{Box}(1+\epsilon))
\end{equation}
Observe now that we may bound the right hand side by $10(1+\epsilon)\delta(\beta^\ast)$ (pick any point $p\in\beta^\ast\cap\operatorname{Box}(1+\epsilon)$; then 
any other point $q\in\beta^\ast\cap\operatorname{Box}(1+\epsilon)$ satisfies $\left|p-q\right|\leq 2(1+\epsilon)\sqrt{1+2^{2/3}+2^{4/3}}<5(1+\epsilon)$, and thus is 
within distance $5(1+\epsilon)\delta(\beta^\ast)$ of $p$ along $\beta^\ast$).  Thus we have:
\begin{equation}
\int_1^{1+\epsilon}\#(\beta^\ast\cap\partial\operatorname{Box}(r))\,dr\leq 10(1+\epsilon)\delta(\beta^\ast)
\end{equation}
Hence there exists $r_1\in(1,1+\epsilon)$ such that:
\begin{equation}
\#(\beta^\ast\cap\partial\operatorname{Box}(r_1))\leq 10(1+\epsilon^{-1})\delta(\beta^\ast)
\end{equation}
Similarly, let us intersect $\beta^\ast$ with the planes $\{z=s\}$ (dividing the long dimension of $\operatorname{Box}(r_1)$ roughly in half), and write:
\begin{equation}
\int_{-\epsilon}^\epsilon\#(\beta^\ast\cap\operatorname{Box}(r_1)\cap\{z=s\})\,ds\leq\operatorname{Length}(\beta^\ast\cap\operatorname{Box}(r_1))\leq 10(1+\epsilon)\delta(\beta^\ast)
\end{equation}
As above, we find $s_1\in(-\epsilon,\epsilon)$ such that:
\begin{equation}
\#(\beta^\ast\cap\operatorname{Box}(r_1)\cap\{z=s_1\})\leq 5(1+\epsilon^{-1})\delta(\beta^\ast)
\end{equation}
For technical conveinence (to satisfy the hypotheses of Lemma \ref{doublebubble2}), we shall also assume that $\{z=s_1\}$ is 
disjoint from $\beta^\ast\cap\partial\operatorname{Box}(r_1)$ (at worst this disqualifies a finite number of values of $s$).

Now define the double bubble $(S^\ast,D^\ast):=(\partial\operatorname{Box}(r_1),\{z=s_1\}\cap\operatorname{Box}(r_1))$, and denote by 
$H_1^\ast$ and $H_2^\ast$ the two half-boxes into which $D^\ast$ divides $S^\ast$.  By construction, we have:
\begin{align}
\left|S^\ast\cap\beta^\ast\right|+2\left|D^\ast\cap\beta^\ast\right|
&=\#(\beta^\ast\cap\partial\operatorname{Box}(r_1))+2\cdot\#(\beta^\ast\cap\{z=s_1\}\cap\operatorname{Box}(r_1))\cr
&\leq 20(1+\epsilon^{-1})\delta(\beta^\ast)
\end{align}
Now if $20(1+\epsilon^{-1})\delta(\beta^\ast)<I(F,\beta)$, then we may apply Lemma \ref{doublebubble2} to $(\varphi^{-1}(S^\ast),\varphi^{-1}(D^\ast))$ 
and conclude that for some $i\in\{1,2\}$ and all $\eta>0$, the surface $F\cap(B(\eta)+\operatorname{int}(\varphi^{-1}(H_i^\ast)))$ has a connected component 
of genus $g$.  Of course, this implies that $F^\ast\cap(B(\eta)+\operatorname{int}(H_i^\ast))$ has a connected component of genus $g$ for some $i\in\{1,2\}$ 
and all $\eta>0$.  By construction, each $H_i^\ast$ is \emph{strictly} contained in a congruent image of $\operatorname{Box}(2^{-1/3}(1+\epsilon)+\frac\epsilon 2)$, 
so we have shown that:
\begin{equation}
20(1+\epsilon^{-1})\delta(\beta^\ast)<I(F,\beta)\implies\left(2^{-1/3}(1+\epsilon)+\frac\epsilon 2\right)\mathcal R\subseteq\mathcal R
\end{equation}
One easily calculates that for $\epsilon=\frac 17$, we have $2^{-1/3}(1+\epsilon)+\frac\epsilon 2<1$, and 
so the right hand side would contradict the fact that $\inf\mathcal R>0$.  Thus we have $160\cdot\delta(\beta^\ast)\geq I(F,\beta)$, as was to be shown.
\end{proof}

\bibliographystyle{plain}
\bibliography{torusdist}

\end{document}